\documentclass[12pt]{article}
\usepackage{amsmath}
\usepackage{amssymb}
\usepackage{amsthm}
\usepackage[final]{pdfpages}
\usepackage{graphicx}
\usepackage{hebcal}

\newcommand{\RR}{\mathbb R}
\newcommand{\CC}{\mathbb C}
\newcommand{\FF}{\mathbb F}
\newcommand{\PP}{\mathbb P}
\newcommand{\be}{\begin{equation}}
\newcommand{\ee}{\end{equation}}
\newcommand{\inn}[1]{\langle #1 \rangle}
\DeclareMathOperator{\tr}{\mathrm{tr}}
\DeclareMathOperator{\diag}{\mathrm{diag}}

\newtheorem{thm}{Theorem}
\newtheorem{prob}{Problem}
\newtheorem{lem}[thm]{Lemma}
\newtheorem{cor}[thm]{Corollary}
\newtheorem{prop}[thm]{Proposition}

\title{Energy Minimization in $CP^n$ Some Numerical and Analytical Results}
\author{Radel Ben Av, Assaf Goldberger, Giora Dula and Yossi Strassler}
\date{\today}

\begin{document}

\maketitle
\section{Introduction}
Minimizing Energy functional has been an active area of search both Numerically and Analytically for centuries.
In recent years many numerical and analytical results have been found \cite{3}
This approach has been proven useful also in the field of finding symmetric geometrical objects \cite{4}. In recent 
years there is a growth of interest in complex spaces. One of the driving forces behind this interest is their a
applicability to Quantum Mechanics.
In particular there is interest in $CP^n$ spaces as wave functions are elements in $CP^n$.

(Quantum) Random Access Codes - (Q)RAC have been defined by \cite{1}. These codes enable a communicating r bits 
using $s<r$ (q)bits, the caveat being that the receiver can retrieve the bits correctly with  probability $p<1$.  For example it has been shown that there is a QRAC with r=2 and s=1 but there is no RAC analogue. Yet 
another result is that QRAC with r=3 is also possible but there is no QRAC with r=4 and s=1 \cite{2}. It can be 
noted that the r=3 code has interesting geometrical properties that will be discussed later.

In the present work we will discuss solutions for a certain optimization problem in $CP^n$. We will provide 
numerical results for a range of parameters. We will also address analytically some sub-ranges and we will 
show that in these cases the numerical results agree with the analytical ones. Moreover we will claim that 
the results indicate the existence of geometrical structures that provide solutions of the optimization problem.

Let $\FF$ denote the field $\RR$ or $\CC$. We endow the vector space $F^n$ with the standard inner product 
and norm given by
$$\inn{v,w}=\sum_{i=1}^n v_i\overline{w_i}, \text{ and } ||v||=\inn{v,v}^{1/2}.$$

We wish to solve the following optimization problem:

\begin{prob}\label{prob:1}
	Given positive integers $n,m,p$, find
	$$M^p(m,n):=\min \sum_{1\le i<j\le m}|\inn{v_i,v_j}|^p\text{ s.t. } \forall\ 1\le l\le m \ : v_l\in F^n \text{ and } 
||v_l||=1.$$
\end{prob}

We think of solutions to this problem as ways to spread out $m$ points as much as possible in the unit sphere,
 more precisely, the projective space $\FF\mathbb P^{n-1}$ i.e. $\CC P^{n-1}$ for $\FF=\CC$. For $m\le n$ the 
problem is trivial, as we can set the vectors to be orthogonal to each other and $M^p(m,n)=0$.\\
\section{Numerical Approach}
In this section we use a numerical approach to solve problem 1. We applied a greedy algorithm along the following lines
    
\begin{enumerate}
    \item set an initial random configuration of n complex vectors $v_i , 1\leq i \leq n $. The components of each 
vector $v_i$ were m complex numbers such that both the real and the imaginary parts were chosen with uniform 
distribution in [0,1]. Each vector was then normalized using division  by the norm.
    \item set initial stepsize - $\delta$
    \item Loop until the stepsize $\delta$ is small enough or the number of sweeps is too big - 
    \begin{enumerate}
        \item Choose a random index  $k,  1\leq k \leq m$.
        \item choose a random component $l, 1\leq l \leq n$.
        \item choose a complex number - z - such that both the real and the imaginary parts are uniformly  
distributed in [-1,1].
        \item add $\delta z$ to the $l$'th component  of $v_k$.
        \item re-normalize $v_k$.
        \item if $M^p(m,n)$ has decreased accept the the suggested change. otherwise discard it.
        \item if too many changes have been accepted - increase $\delta$.
        \item if too few changes have been accepted - decrease $\delta$.
        
    \end{enumerate}
    \item Output the results - the value of $M^p(m,n)$, the value of all the final vectors $v_i$ and the values 
of $|<v_i,v_j>|$ for all the pairs $i,j$.
\end{enumerate}

Clearly this is not the best optimization algorithm, e.g. Newton-Raphson could be implemented. However it was easily and readily available to us.

For each value of $p,m,n$ we ran the minimization several times and the resulting optimum $M^p(m,n)$ was stable for p=2 and p=4. For p=6 and $n>7$ the minimum value of $M^p(m,n)$ was stable only up to the first 6 digits. Further investigation for this issue is required.
Nevertheless,  the final configuration was in many cases not the same. This indicates that the minimal $M^p(m,n)$ is (almost always) unique 
but the solution space is of higher dimension. For some values of (m,n,p) the value of $|\inn{v_i,v_j}|$ was unique. Moreover in some cases $|\inn{v_i,v_j}|=C$ for all $i,j$ where C is a function of (m,n,p). These cases are actually simplexes in $CP^n$. As can be seen it can occur that a simplex solution is probably the only solution for (m,n,p) while being only a point in the solution space for (m,n,p')  $p'\neq p$.\\

The resulting $M^p(m,n)$ are presented in the following tables. Table 1 presents the results for {\bf p=2}, 
Table 2 for {\bf p=4} and Table 3 for {\bf p=6}. The simplex cases are indicated with yellow background. The convergence rate was dependent on the values of $p,m,n$ 
in a non-trivial way. We have not addressed this issue yet.

All the results in table 1 seem very elegant. Indeed in the following section we will provide analytic solution 
that coincides with the numerical results. Moreover some of the results (e.g, $M^4(5,10)$) are also intriguing.   

Visual inspection indicated that the values of the minimal Energy for a given p and m tend to behave quadratically for large n in p=2, p=4. In order to check it we provide the approximate second derivative $D^p(m,n)$ as a function of n.
$$D^p(m,n)=M^p(m,n+1)-2M^p(m,n)+M^p(m,n-1)$$
The results are shown in tables 4 and 5.

This conjecture is evidently true for p=2 as can be seen from the analytical results in section 3.1. For p=4 the numerical result is accordance with the conjecture of equdistribution of vectors for large n as can be seen in 3.2. For p=6 and n=2 it also seems that the value is purely quadratic not only asymptotically but starting in finite n. For p=6 and $n>2$ the large m limit might have not been reached yet.

\begin{table}
		\includegraphics[trim={1cm 3.5cm 0 0},scale=0.8]{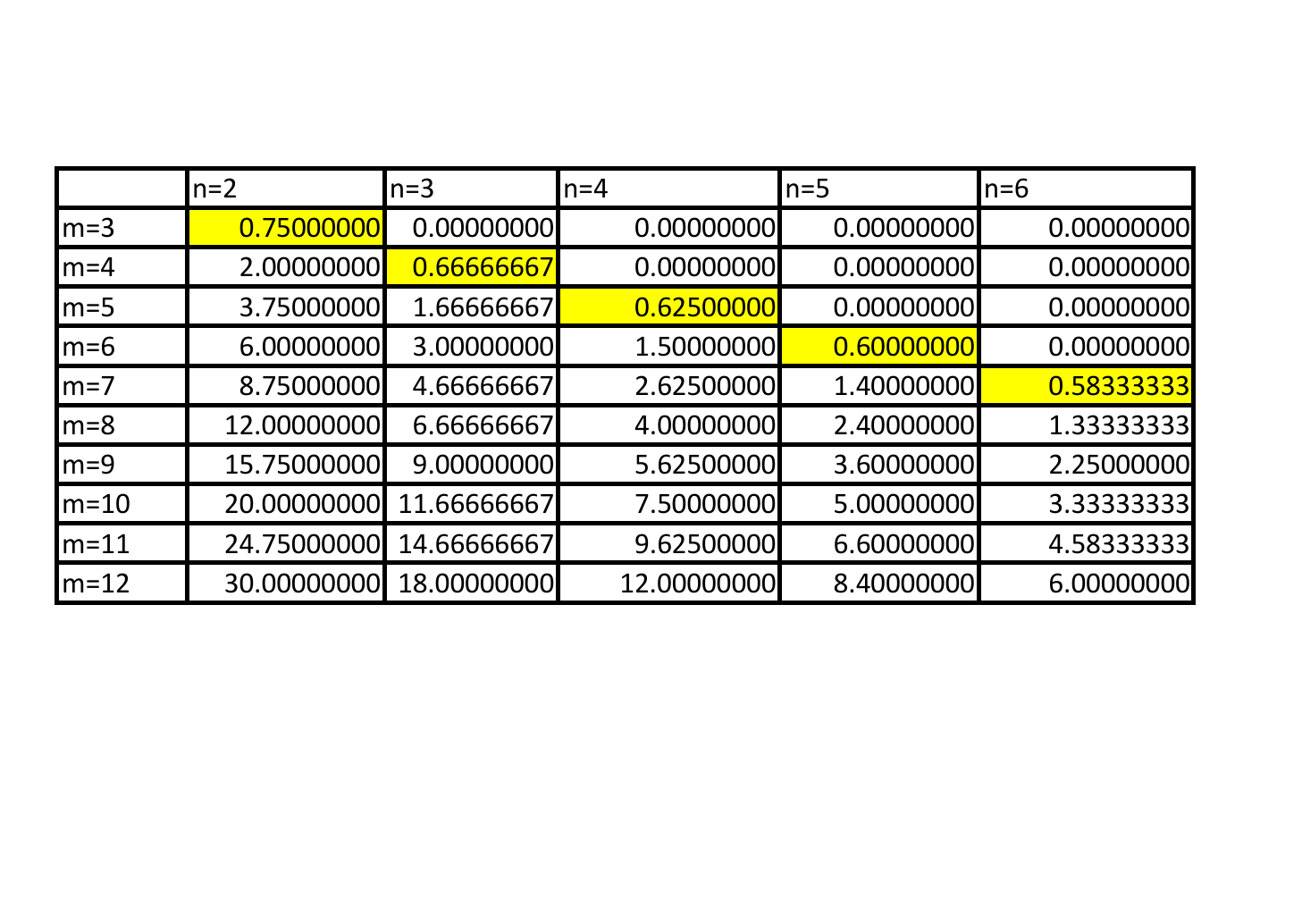}
\caption{ Minimum value for p=2}
\end{table}
\begin{table}
	
	\includegraphics[trim={1cm 2cm 0 0},scale=0.8]{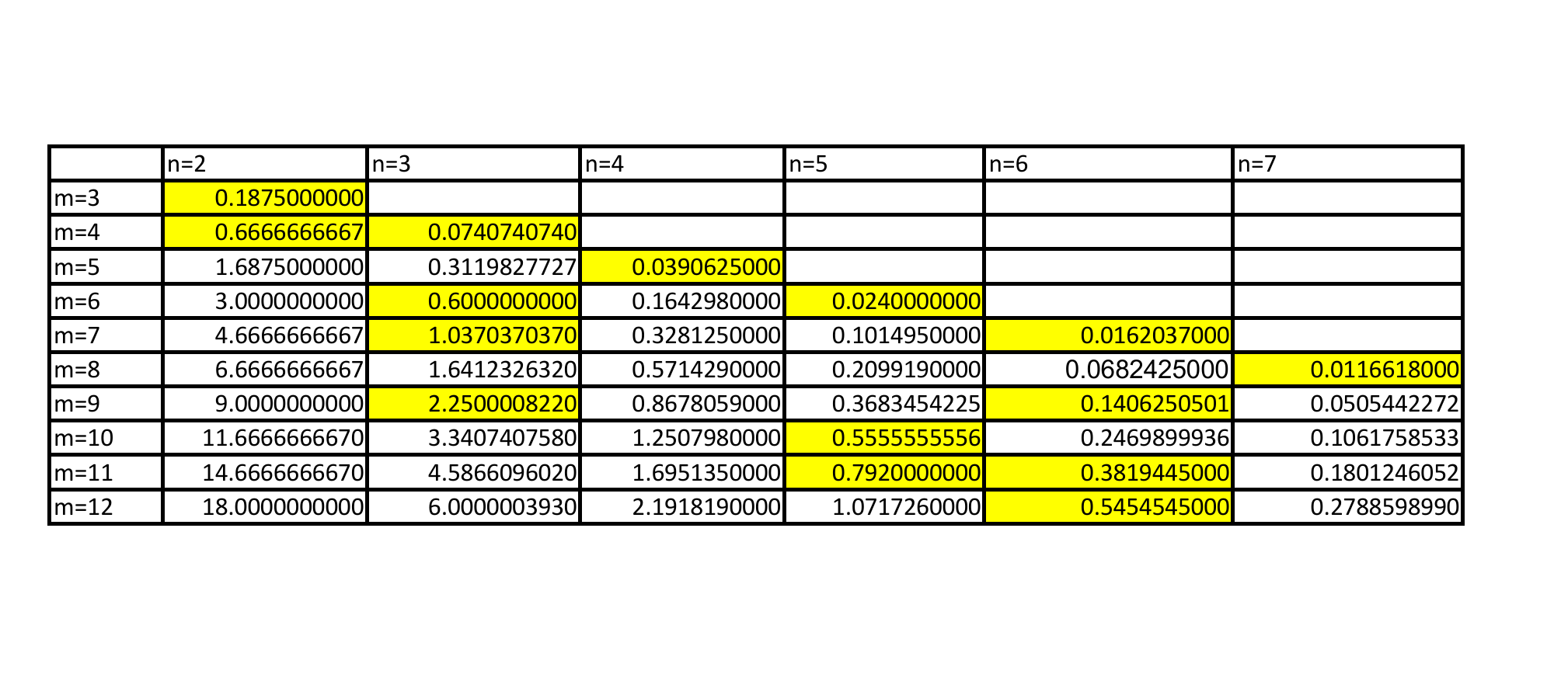}
	\caption{Minimum value for p=4}
\end{table}

\begin{table}
\includegraphics[trim={1cm 2cm 0 0},scale=0.8]{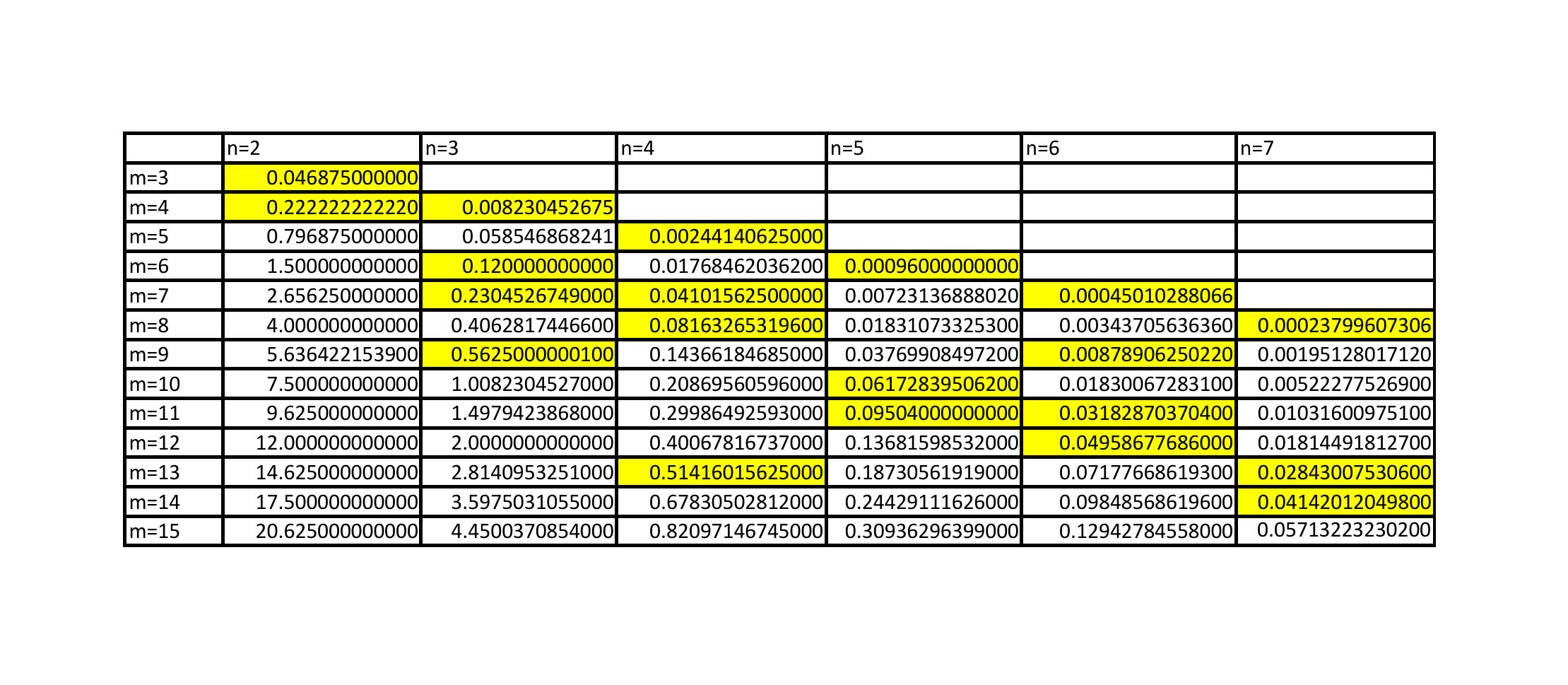}
\caption{Minimum value for p=6}
\end{table}

\begin{table}
	\includegraphics[trim={1cm 2cm 0 0},scale=0.8]{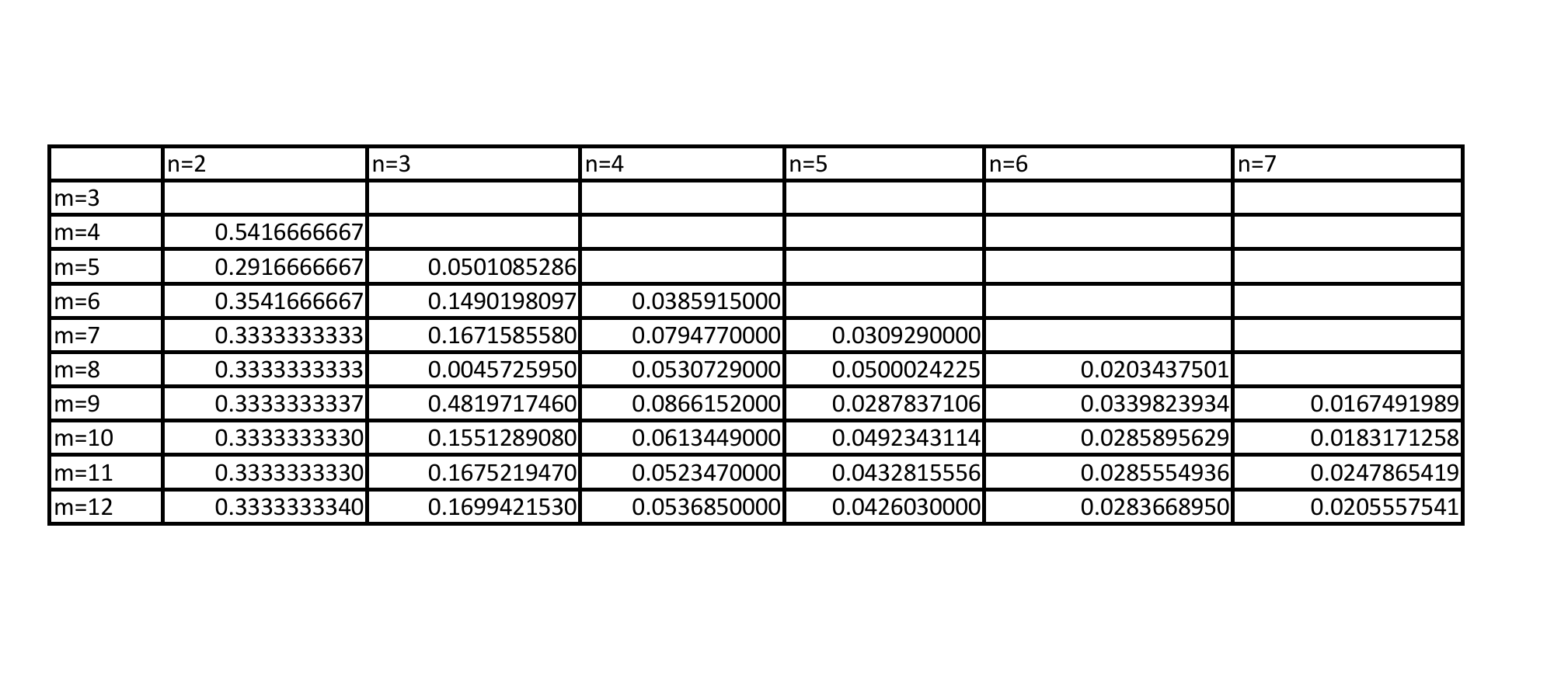}
	\caption{Numerical 2nd Derivative for p=4}
\end{table}

\begin{table}
	\includegraphics[trim={1cm 2cm 0 0},scale=0.8]{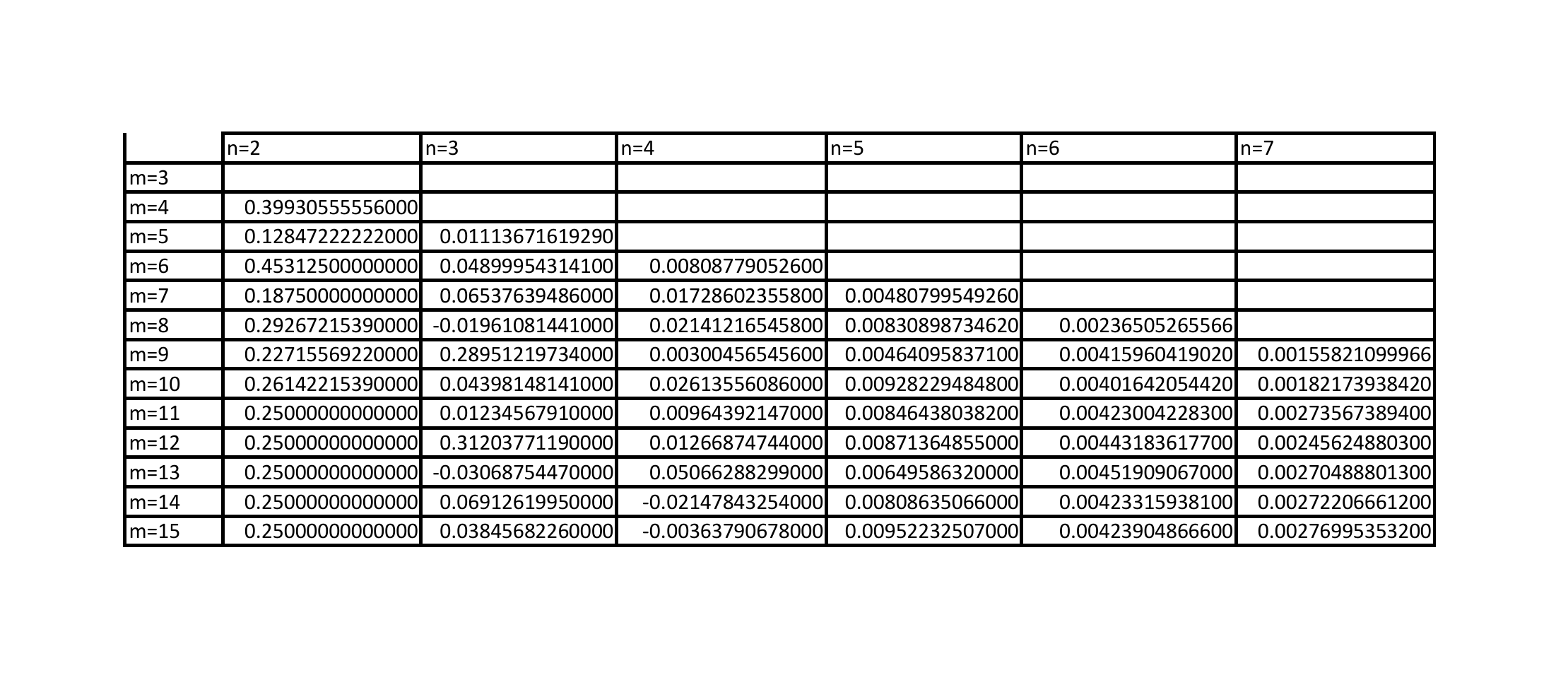}
	\caption{Numerical 2nd Derivative for p=6}
\end{table}

\section{Analytic Approach}
\subsection{p=2 Solution}

In order to solve this problem, we introduce a new problem which is in some way a relaxation of Problem \ref{prob:1}.

\begin{prob}\label{prob:2}
	Given positive integers $n<m$, find
	$$P(m,n):= \min \sum_{1\le i,j\le m}|\inn{v_i,v_j}|^2\text{ s.t. } \forall \ 1\le k \le m \ v_k\in F^n \text{ and } \sum_{k=1}^m ||v_k||^2=m.$$
\end{prob}
Every solution $(v_i)$ to Problem \ref{prob:1} is within the feasible region of Problem \ref{prob:2} and gives 
value of $2M(m,n)+m$ to its objective function. Therefore,
\be\label{eq:1} P(m,n)\le 2M(m,n)+m.\ee We will show below that there is a solution to Problem \ref{prob:2} within the 
feasible region of Problem \ref{prob:1}, which  will turn \eqref{eq:1} into an equality.\\

We turn now to the solution of Problem \ref{prob:2}. Every collection $(v_i)$ of $m$ vectors in $F^n$, will be encoded 
as a $m\times n$ matrix $V$ with $v_i$ as the $i$th row. The condition $\sum ||v_i||^2=m$ becomes $\tr(VV^*)=m$. The 
entries of $VV^*$ are the inner products $\inn{v_i,v_j}$ so the objective function becomes $\tr((VV^*)^2)$. We have the 
following equivalent formulation to Problem \ref{prob:2}.

\begin{prob}\label{prob:3}
	Given positive integers $n<m$, find
	$$P(m,n)=\min \tr(VV^*VV^*) \text{ s.t. } V\in F^{m\times n} \text{ and }\tr(VV^*)=m.$$
\end{prob}

\begin{proof}[Solution]
	One has $\tr(VV^*)=\tr(V*V)$ and $\tr(VV^*VV^*)=\tr((V^*V)^2)$. As $Q=V^*V$ ranges over all positive semidefinite
 Hermitian $n\times n$ matrices of trace $m$, we need to find the minimum of $\tr(Q^2)$ over all such matrices. As $\tr$ 
is unchanged under matrix conjugation, it is sufficient to restrict attention just to diagonal positive semidefinite 
matrices $Q$. Let $Q=\diag(\lambda_1,\ldots,\lambda_n)$. Then our problem is equivalent to finding the minimum of 
$\sum_{i=1}^n \lambda_i^2$ subject to $\lambda_i\ge 0$ and $\sum_{i=1}^n \lambda_i=n$. Clearly the minimum is achieved 
for equal $\lambda_i=m/n$, where $\sum \lambda_i^2=m^2/n.$ Rolling back, $P(m,n)=\min \tr(VV^*VV^*)=m^2/n$, and the 
minimizing $V$ can be taken to be any $m\times n$ matrix such that $V^*V= (m/n)I_n$. 
\end{proof}

We see that a solution to Problems \ref{prob:2} and \ref{prob:3} is obtained by a $m\times n$ matrix $V$ with orthogonal 
columns. Thus as one solution we can simply to take 
$$V_0=\sqrt{\frac{m}{n}}\begin{bmatrix}
I_n\\ \hline \\ 0_{m-n \times n}
\end{bmatrix},$$
and the general solution is $V=UV_0$, as $U$ ranges over the set of all unitary (orthogonal) matrices over 
$\mathbb C$($\mathbb R$). Our next step is to show that there exists a solution $V_1$ with all rows equal norm 
(necessarily $1$). This follows from the following Lemma.

\begin{lem}\label{lem:1}
	For every matrix $W\in F^{m\times n}$, there exists a unitary (orthogonal) matrix $U\in F^{m\times m}$ 
such that all rows of $UW$ have equal norm.
\end{lem}

\begin{proof}
	We define the continuous function
	$$F(U)=\max_i ||(UW)_i||.$$ Since $U$ ranges over a compact set, then $F$ attains a minimum, at some point 
$U_1$. Furthermore, we may assume that the number $q$ of rows of $U_1W$ with the maximal norm $F(U_1)$ is the minimum 
possible. If $q=n$, then we are done. Otherwise, we will derive a contradiction. Suppose that $q<n$, and there are two 
rows, $r_i=(U_1W)_i$ and $r_j=(U_1W)_j$ with $||r_i||=F(U_1)>||r_j||$. Without loss of generality, let $i=1$ and $j=2$. 
Modify $U_1$ to $U_2(\phi)$  given by
	$$U_2(\phi)=\begin{bmatrix}
	\cos \phi & \sin \phi & 0\\
	-\sin \phi & \cos \phi &0\\
	0 & 0 & I_{m-2}
	\end{bmatrix} U_1.$$
	Then only the first two rows of $U_2(\phi)U_1W$ vary as functions of $\phi$. For $\phi=0$ we just get $U_1W$. 
But for $\phi=\pi/2$ the first two rows are swapped (and the second one is being multiplied by $-1$). It follows 
that for small values of $\phi$, the first two rows will have norm strictly smaller than $F(U_1)$. If $q>1$, then 
we found a new matrix $U_2(\phi)U_1W$ with $F(U_2(\phi)U_1W)=F(U_1)$, but with smaller $q$. If $q=1$, then $F(U_2(\phi)U_1W)<F(U_1)$. 
In both cases we obtain a contradiction, and the lemma is proved.
\end{proof}

\begin{cor}\label{cor:2}
	We have $P(m,n)=2M(m,n)+m=m^2/n$ and
	\be \label{eq:2} M(m,n)=\frac{m(m-n)}{2n}\ \ (m>n).
	\ee 
\end{cor}
Notice that for $m\le n$ we have $M(m,n)=0$, as we can choose the rows of $V_0$ to be part of the standard basis. 
The proof of Lemma \ref{lem:1} gives us an efficient algorithm for solving Problem \ref{prob:1}. We actually see that there 
are many solutions, because the dimension of the unitary (orthogonal) group is greater than $m$. In some cases we can obtain 
a solution which is a \emph{simplex}. This means that in addition $|\inn{v_i,v_j}|$ has some constant value for all $i\neq j$.
We have

\begin{prop}
	In a simplex solution for  Problem \ref{prob:1} we have
	$$|\inn{v_i,v_j}|^2=\frac{m-n}{n(m-1)} \text{ for all } i\neq j.$$
	Furthermore, there exists a simplex solution with parameters $(m,n)$, if and only if there exists a simplex solution 
in parameters $(m,m-n)$.
\end{prop}

\begin{proof}
	The first assertion follows easily from \eqref{eq:2}. If $V$ is the matrix corresponding to a simplex solution with 
parameters $(m,n)$, then we may complete $V$ to an $m\times m$ unitary (orthogonal) matrix $\hat V$ and the complement submatrix 
is a simplex solution in parameters $(m,m-n)$.
\end{proof}

\section{The complex cases $p=4$ and $p=6$ at $n=2$}
When we restrict to $n=2$ over $F=\mathbb C$, we are able to understand the cases $p=4$ (Quad) and $p=6$ (Hex) at least in part. We exploit the fact that there is a topological identification  $\CC \PP^1\simeq S^2$. Under this identification the Quad and Hex complex problems essentially reduce to the real square problem, plus some extra conditions which can be satisfied for $m$ large enough, at least for $m$ even.\\

First, let us recall the isomorphism $\CC \PP^1\simeq S^2$. Let $\mathbb H$ denote the quaternion algebra over $\RR$ with basis $1,i,j,k$ and relations $i^2=j^2=k^2=-1$ and $ij=-ji=k,jk=-kj=i,ki=-ik=j$. We think of $\mathbb H$ as a two dimensional vector space over $\CC$ with basis $1,j$, endowed with the standard Hermitian form $\inn{}_\CC$. We view the 3-sphere $S^3$ as the subset $\mathbb H_1$ of all elements of norm $1$. We identify $S^2$ as the 'equator' $Im\mathbb H=\{ yi+zj+wk\ | y^2+z^2+k^2=1\}.$ This subset is the conjugation orbit of $i$ under the action of the quaternion group $\mathbb H^\times$. The centralizer of $i$ in $\mathbb H^\times$ is $\CC^\times$, and the conjugation on $i$ supplies us a topological homeomorphism 
$$\CC\PP^1=\mathbb H^\times/\CC^\times=S^3/S^1 \simeq Im\mathbb H=S^2.$$ Write this map as $S:\CC\PP^1 \to S^2,\ u\mapsto uiu^{-1}$. It is useful to give a comparison between the metrics on both. We have

\begin{lem}\label{7}
	\be |\langle u,v\rangle_\CC |^2=\frac{1+\langle Su,Sv\rangle_\RR}{2},\ee
	and in terms of angles,  $|\langle u,v\rangle_\CC |=\cos\frac{\phi}{2}$ iff  $\langle Su,Sv\rangle_\RR=\cos\phi.$
\end{lem}

\begin{proof}
	The real product $\inn{-}_\RR$ on $S^2$ is the restriction of the real product on $\mathbb H$ given by $\inn{u,v}_\RR=Re\inn{u,v}_\CC$. The multiplication on $\mathbb H$ on left and right is unitary w.r.t to the hermitian product, hence $\inn{Su,Sv}_\RR=\inn{S1,S(u^{-1}v)}_\RR$ and $\inn{u,v}_\CC=\inn{1,u^{-1}v}_\CC$. So it is sufficient to prove the lemma for $u=1$. We have $S1=i$ and $\inn{S1,Sv}_\RR=Re\inn{i,viv^{-1}}_\CC=Re\inn{iv,vi}_\CC$. If $v=a+bi+(c+di)j$ is such that $a^2+b^2+c^2+d^2=1$, then $Re\inn{iv,vi}_\CC=a^2+b^2-c^2-d^2=2(a^2+b^2)-1$. On the other hand, $|\inn{1,v}_\CC|^2=a^2+b^2$. The lemma follows.
\end{proof}

\subsection{The case $p=4$ and $n=2$.}
In view of Lemma \ref{7}, for $p=4$ it suffices to solve on $S^2$ the following problem.
\begin{prob}\label{p=4}
	Find vectors $v_1,\ldots,v_m\in S^2$ that minimize the quantity
	$$Q(v_1,\ldots,v_m)=\sum_{i,j} \frac{1+2\inn{v_i,v_j}_\RR+\inn{v_i,v_j}_\RR^2}{4}.$$
\end{prob}
A solution to the original problem \ref{prob:1} will be obtained by $S^{-1}v_1,\ldots,S^{-1}v_m$.
The point is that the linear term $\sum_{i,j}\inn{v_i,v_j}_\RR=||\sum_i v_i||^2\ge 0$, and is $0$ if and only if $\sum_i v_i=0$. Thus it makes sense to formalize the following problem: 

\begin{prob}
	\label{p=4,0} Find vectors $v_1,\ldots,v_m\in S^2$ that minimize the function
	$$P(v_1,\ldots,v_m)=\sum_{i,j} \inn{v_i,v_j}_\RR^2,$$ and in addition satisfy $\sum_i v_i=0.$
\end{prob}
The following lemma is clear.
\begin{lem}  
Any solution to Problem \ref{p=4,0} is necessarily a solution to Problem \ref{p=4}. Conversely, if a solution to Problem \ref{p=4,0} exists, then all solution to Problem \ref{p=4} are also solutions to Problem \ref{p=4,0}. $\Box$
\end{lem}
We will prove now  
\begin{thm}
	\begin{itemize}
		\item[(a)] For $m\ge 6$ or $m=4$ there is always a solution to Problem \ref{p=4,0}.
		\item[(b)] For $m\ge 6$ or $m=4$, $$M^4(m,2)=\frac{m(m-3)}{6} .$$
	\end{itemize}
	
\end{thm}
\begin{proof}
	The tuples $(v_1,\ldots,v_m)\in (S^2)^m$ which minimize $\sum_{i,j}\inn{v_i,v_j}_\RR$ are exactly the tuples  $(v_1,\ldots,v_m)$, which when we arrange them as a matrix $V$ with rows $v_i$, the columns of $V$ are orthogonal. All we need is to show that we can find such $V$, satisfying the extra condition that $\sum_i v_i=0$.\\
	
	Let $m\ge 6$ and $\phi_k=2k\pi/m$, $k=0,1,\ldots,m-1$. We will construct $$v_k=[\cos2\phi_k,\sin2\phi_k\cos\phi_k,\sin2\phi_k\sin\phi_k].$$
	clearly $v_k$ are normalized. To prove that $\sum_k v_k=0$ and that the columns of $V$ are orthogonal, it is best to rewrite $\cos\phi=(e^{i\phi}+e^{-i\phi})/2$ and $\sin\phi=(e^{i\phi}-e^{-i\phi})/2i$. It is then seen that all computations involve sums $\sum_k e^{2jk\pi i/m}$ with $-5\le j \le 5$ and $j\neq 0$, hence these sums are $0$ as long as $m\ge 6$.
	For $m=4$ there exists a Hadamard $4\times 4$ matrix whose first column is $[1,1,1,1]^T$. Thus we may take $V$ to be the remaining $3$ columns normalized by a factor of $1/sqrt{3}$.\\
	
	For any solution to problem \ref{p=4,0} we have that $P(v_1,\ldots,v_m)=m^2/3$ (cf. Corollary \ref{cor:2}) and so $Q(v_1,\ldots,v_m)=m^2/4+m^2/12=m^2/3$. By changing the sum in $Q$ to sum over $i<j$ we obtain 
	$ 2M^4(m,2)+m=m^2/3$ which implies (b).
	
\end{proof}

\subsection{The case $p=6$ and $n=2$}.
Using the isomorphism $S$ we are able to analyze the case $p=6$ and $n=2$, at least when $m$ is even. The functional that we have to minimize is
$$Q'(v_1,\ldots,v_m)=\sum_{i,j}\frac{1+3\inn{v_i,v_j}_\RR+\inn{v_i,v_j}_\RR^3+3\inn{v_i,v_j}_\RR^2}{8}.$$
We have the following key observation.
\begin{lem}
	For every choice of vectors $v_1,\ldots,v_m\in \RR^k$, and in integer $r\ge 1$, $\sum_{i,j}\inn{v_i,v_j}_\RR^r\ge 0$.
\end{lem}
\begin{proof}
	The matrix $G=(\inn{v_i,v_j}_\RR)_{i,j}$ is symmetric positive semidefinite. It is well known that the Hadamard (=pointwise) product of symmetric positive semidefinite matrices is again symmetric positive semidefinite. Therefore the powers $G^{(r)}=(\inn{v_i,v_j}_\RR^r)_{i,j}$ are positive semidefinite. Hitting this matrix from both sides by the vector of 1's proves the lemma.
\end{proof}

In light of the lemma, it suffices to produce solution to the following problem:
\begin{prob}\label{prob:6}
	Find solutions $v_1,\ldots,v_m$ that minimize the function
	$$P(v_1,\ldots,v_m)=\sum_{i,j} \inn{v_i,v_j}_\RR^2,$$
	and in addition satisfy $\sum_{i,j}\inn{v_i,v_j}_\RR=\sum_{i,j}\inn{v_i,v_j}_\RR^3=0$.
\end{prob}
The existence of a solution to Problem \ref{prob:6} will give rise to the value of the function $M^6(m,2)$.

\begin{thm}
	\begin{itemize}
		\item[(a)] For all even $m\ge 6$, a solution to Problem \ref{prob:6} exists.	
		\item[(b)] For all even $m\ge 6$, 
		$$M^6(m,2)=\frac{m(m-4)}{8}.$$
	\end{itemize}

\end{thm}

\begin{proof}
	Let $m=2r$, $r\ge 3$ and consider an $r\times 3$ matrix $W$ whose columns are orthogonal, and whose rows are normalized. We form the matrix $V=[W;-W]$ whose top $r$ rows are those of $W$, and bottom $r$ rows are their negatives. Let $v_1,\ldots,v_m$ be the rows of $V$. Then it is clear that  $\sum_{i,j}\inn{v_i,v_j}_\RR=\sum_{i,j}\inn{v_i,v_j}_\RR^3=0$, because every product appears twice with each a positive and a negative sign. Also, our choice minimizes $P$, hence we have a solution to Problem \ref{prob:6}.\\
	
	According to Corollary \ref{cor:2}, $P(v_1,\ldots,v_m)=m^2/3$, and $Q'(v_1,\ldots,v_m)=m^2/4$.
	Using $2M^6(m,2)+m=Q'=m^2/4$ we obtain part (b) of the theorem.
	
\end{proof}

\section{Asymptotic Equidistribution Estimates in the complex case}
We turn to the more general problem of finding
$$M_p(m,n):=\min \sum_{1\le i<j\le m}|\inn{v_i,v_j}|^{2p}\text{ s.t. } \forall 1\le i\le m \ v_i\in \mathbb C^n \text{ and } ||v_i||=1.$$
It has been observed experimentally that $M_p(m,n)$ behaves quadratically in $m$, at least for $m$ large enough. That is, 
$$M_p(m,n)=A_2(p,n)m^2+A_1(p,n)m+A_0(p,n), \ \ \ m\gg 0.$$
In what follows we shall perform an asymptotic calculation which will support the numerical values of the leading coefficients 
$A_2(p,n)$ discovered by experiments.\\

Our expectation will be that in a minimal configuration when $m$ is large, the points $v_i$ are equidistributed along $S^{2n-1}/S^1$. 
It might be possible to support this (intuitive) assumption by some calculus of variations and some asymptotic bounds, but we will 
not do it for now. By the assuming asymptotic equidistribution and quadratic behavior, we arrive at the relation 
$$A_2(p,n)=2E(|\langle u,v \rangle|^{2p}),$$
where the expectation is taken over all  $u,v\in S^{2n-1}/S^1$ with respect to the Fubini-Studi measure.\\

For computing the expectation, without loss of generality we can fix $u=(1,0,\ldots,0)$, and let $v=(z_1,\ldots,z_n)$ vary. 
As a consequence, $E(|\inn{u,v}|^{2p})=E(|z_1|^{2p})$, where $v$ runs over $S^{2n-1}/S^1$.
The Fubini-Studi form is the 2-form given on $S^{2n-1}$ by $\omega=\sum_k dz_k\wedge d\overline{z_k}.$ This form is invariant 
under phase multiplication, hence descends to (a symplectic form  on) $S^{2n-1}/S^1$. The Fubini-Studi measure is given by 
$d\Phi=\omega^{n-1}$. By passing to polar coordinates, $z_k=r_k\exp(i\theta_k)$, we can rewrite $\omega$ as
$$\omega=\sum_k r_kdr_k\wedge d\theta_k=\frac{1}{2}\sum_k dt_k\wedge d\theta_k, \ \ t_k=r_k^2.$$
We have a map $\sigma: S^{2n-1}/S^1 \to \Delta$ where $\Delta=\{(t_1,\ldots,t_n)| t_k\ge 0,\ \ \sum_kt_k=1\}$ is the 
standard $n-1$-simplex given by $\sigma(z_1,\ldots,z_n)=(|z_1|^2,\ldots,|z_n|^2).$ 
The pushforward of the measure $\omega^{n-1}$ to $\Delta$ by this map, becomes $\lambda=(2\pi)^{n-1}\eta^{n-1}$ for 
$\eta=\sum_k dt_k$. Moreover, $\lambda=\eta^{n-1}$ is just the Lebesgue measure on the simplex $\Delta$. It follows that 
$E(|z_1|^2)=E(t_1)$ where $t=(t_1,\ldots,t_n)$ runs uniformly on $\Delta$ w.r.t. the Lebesgue measure.\\ 

It is easy now to compute $E(t_1^p)$. This is given by the integral quotient of
$$E(t_1^p)=\frac{\int_0^1 t_1^p (1-t_1)^{n-2}dt_1}{\int_0^1 (1-t_1)^{n-2}dt_1}.$$
The reason for the $(1-t_1)^{n-2}$ factor is that once we fixed $t_1$, then $(t_2,\ldots,t_n)$ run uniformly on an 
$n-2$-simplex with sum $1-t_1$, which has volume proportional to $(1-t_1)^{n-2}$. We thus have
$$\frac{1}{2}A_2(p,n)=E(t_1^p)=\frac{B(p+1,n-1)}{B(1,n-1)}=\frac{\Gamma(p+1)\Gamma(n-1)\Gamma(n)}{\Gamma(1)\Gamma(n-1)
\Gamma(p+n)}=\frac{p!(n-1)!}{(n+p-1)!}.$$
This value is supported by our experiments.

\section{Summary}

In this work we have discussed finding the minimum energy for a certain family of functions in $CP^n$. The justification for this functional relates to finding a set of vectors that are maximally distant from each other. We have evaluated this value numerically. The numerical results indicated that there could be a geometrical interpretation to these configuration (as was the original motivation) We have shown that the for some subsets of this family the minimum can be evaluated analytically and the results agree with the numerical results.

Many interesting phenomena seem to be underneath the results reported here. For example, in p=6 and n=2, the minimal value seems to be rational right from m=3. Some of the minimal configurations are related to geometrical structures that are extensions of platonic objects in $CP^n$. The case of n=2 correspond to all the classical platonic objects in real D=3. There are still many open questions as to what is the geometrical nature of all the minimal configurations. In some cases the numerical results hints at a deeper connections to structures of higher symmetry in these spaces.

\end{document}